\documentclass[10pt]{amsart}
\usepackage{a4wide}
\usepackage{amssymb}
\usepackage{mathrsfs}
\usepackage{xspace}
\usepackage[all]{xy}
\theoremstyle{plain}
\newtheorem{prop}{Proposition}
\newtheorem*{prop*}{Proposition}

\newtheorem*{thm*}{Theorem}
\newtheorem{cor}[prop]{Corollary}

\newtheorem*{convention*}{Convention}
\theoremstyle{definition}
\newtheorem*{defn*}{Definition}
\newtheorem{defn}[prop]{Definition}
\newtheorem{rem}[prop]{Remark}

\newtheorem*{scholium*}{Scholium}
\newtheorem{example}[prop]{Example}
\newtheorem*{example*}{Example}
\numberwithin{equation}{section}

\newcommand{\fhi}{\varphi}

\newcommand{\sA}{\mathscr{A}}

\newcommand{\sBI}{\mathscr{B_\mathscr{I}}}
\newcommand{\sC}{\mathscr{C}}
\newcommand{\sI}{\mathscr{I}}
\newcommand{\sL}{\mathscr{L}}

\newcommand{\ptens}{\mathrel{\hat\otimes}}
\newcommand{\itens}{\mathrel{\check\otimes}}

\newcommand{\norms}[1]{\left \lVert #1 \right \rVert}

\newcommand{\bid}{^{**}}
\newcommand{\uc}[2]{\mathrm{uc}_{#1}[#2]}
\newcommand{\tc}{{\upshape\textrm{C}}\xspace}
\newcommand{\tm}{{\upshape\textrm{M}}\xspace}
\newcommand{\hb}{\mathrm{H}_\mathrm{b}}
\newcommand{\inv}{^{-1}}
\newcommand{\se}{\subseteq}
\newcommand{\lra}{\longrightarrow}

\begin{document}
\title{A note on topological amenability}
\author[Nicolas Monod]{Nicolas Monod$^\ddagger$}
\address{EPFL, Switzerland}
\thanks{$^\ddagger$Supported in part by the Swiss National Science Foundation}
%
\begin{abstract}
We point out a simple characterisation of topological amenability in terms of bounded cohomology, following Johnson's reformulation of amenability.
\end{abstract}
\maketitle

\section{Introduction}
\subsection{Motivation}
According to Johnson--Ringrose~\cite[2.5]{Johnson}, a locally compact group $G$ is amenable if and only if the (continuous) bounded cohomology $\hb^n(G, E^*)$ vanishes for all $n\geq 1$ and all Banach $G$-modules $E$. In Johnson's terminology, $\hb^n(G,-)$ is the cohomology of the Banach algebra $\ell^1 G$. The importance of this reformulation is that it led to Johnson's definition of amenable algebras. Here $E^*$ denotes the dual module; passing to duals is consubstantial with the nature of amenability as a limiting property, and indeed the corresponding condition on \emph{all} modules simply characterises compact groups instead.

\medskip

This purpose of this note is to record the corresponding characterisation of \emph{amenable actions} in the topological sense; one motivation is the particular case of \emph{exactness} of discrete groups. Definitions will be recalled in the text as needed.

\medskip
This note grew out of my reading of~\cite{Douglas-NowakVAN} where a partial result for the case of exactness is obtained by another approach. I have been informed by Brodzki, Niblo, Nowak and Wright that they have previously obtained a characterisation of amenable actions. After completing this note I received their preprint containing similar ideas; the reader is encouraged to consult that article as well.

\subsection{Topological amenability}
Let $G$ be a group acting on the compact space $X$. We always assume that $X$ is Hausdorff and the action is by homeomorphisms. There is no topology on $G$.

\medskip

What is desired is a cohomological characterisation of the amenability of this action. It should extend the Johnson--Ringrose criterion for the amenability of (discrete) groups. We propose therefore to consider the vanishing of $\hb^\bullet(G, E^*)$ for a class of Banach $G$-modules $E$ determined by $X$. The condition imposed by $X$ should in be void when $X$ is a point, and be functorial. Before going any further, there is an obvious guess: all modules of the form $E=C(X, V)$ where $V$ is any Banach $G$-module. It turns out that there is indeed an easy criterion along those lines:

\medskip

\indent\indent\indent\begin{minipage}{0.8\linewidth}
\itshape
The $G$-action on $X$ is topologically amenable if an only if $\hb^n(G, C(X, V)\bid)$ vanishes for every Banach $G$-module $V$ and every $n\geq 1$.
\end{minipage}

\medskip
\noindent
An aesthetic drawback is that we passed here to the bidual $E\bid$, whilst one expects a criterion involving all duals (this affect also functoriality). We shall therefore move beyond this obvious first guess and return to general Banach $G$-modules $E$. The influence of $X$ will simply be through the data of a $C(X)$-module structure, which should of course be compatible with the $G$-action. Let us therefore call $E$ a \emph{$(G,X)$-module} if it is both a Banach $G$- and $C(X)$-module, such that $g\fhi g\inv$ corresponds to the action of $g\in G$ on $\fhi\in C(X)$. This algebraic condition is too general to capture the analytic situation (and Example~\ref{ex:GX} illustrates that it \emph{cannot} characterise amenability); this reflects the fact that not all $(G,X)$-modules are modules for the underlying crossed product algebra introduced below. Analysing purely the $C(X)$-structure alone, we can however isolate a relevant concept inspired by Kakutani's classical work~\cite{Kakutani41a,Kakutani41b}:

\begin{defn}
We say that a $C(X)$-module $E$ is \emph{of type}~\tm if for all $u\in E$ and $\fhi_i \in C(X)$ with $\fhi_i\geq 0$ one has
$$\| \fhi_1 u\| + \cdots + \| \fhi_n u\| \ \leq\ \| \fhi_1 + \cdots + \fhi_n \| \cdot \|u\|. \leqno{\text{(M)}}$$
We say that $E$ is \emph{of type}~\tc if for all $u_i \in E$ and $\fhi_i\geq 0$ one has
$$\norms{\fhi_1 u_1 + \cdots + \fhi_n u_n} \ \leq\ \| \fhi_1 + \cdots + \fhi_n \| \cdot \max_i\norms{u_i}. \leqno{\text{(C)}}$$
\end{defn}

\noindent
(These notions are mutually dual, yet both types include modules that are not dual. The fundamental examples motivating the terminology are modules of measures, respectively of continuous functions.)

\smallskip

We now have the criterion that we sought:

\begin{prop*}[Concise version]
Let $G$ be a group acting on the compact space $X$. The following are equivalent.
\begin{enumerate}
\item The $G$-action on $X$ is topologically amenable.
\item $\hb^1(G, E^*)=0$ for every $(G,X)$-module $E$ of type~\tm.
\item $\hb^n(G, E^*)=0$ for every $(G,X)$-module $E$ of type~\tm and every $n\geq 1$.
\end{enumerate}
\end{prop*}

More can be said:

\medskip

a.~In ergodic theory, amenable actions (in Zimmer's sense) are characterised by the \emph{relative injectivity} of the corresponding modules of $L^\infty$-maps~\cite{Burger-Monod3,Monod}. We provide a similar statement for topological amenability. However, the notion of ``bounded measurable maps'' becomes problematic when $X$ is non-metrisable, as it will definitely be for applications to exactness where $X$ is the \v{C}ech--Stone space $\beta G$. We submit that a good analogue is provided in complete generality by the dual of the space $\sI(C(X), W)$ of \emph{integral operators} (valued in any Banach space $W$). Hence the statements~\eqref{pt:van:int} and~\eqref{pt:inj:int} below.

\medskip

b.~As in the Johnson--Ringrose criterion, we transit through a characterisation of amenability in terms of invariant means. The equivalence of this with the definition of amenability is routine, but we mention it for completeness,~\eqref{pt:mean:norm} and~\eqref{pt:mean:pos} below.

\begin{prop*}[Long version]\label{thm:amenability:long}
Let $G$ be a group acting on the compact space $X$. The following are equivalent.
\begin{enumerate}
\item The $G$-action on $X$ is topologically amenable.\label{pt:amen}
\item $\hb^n(G, C(X, V)\bid)=0$ for every Banach $G$-module $V$ and every $n\geq 1$.\label{pt:van:cont}
\item $\hb^n(G, \sI(C(X), W)^*)=0$ for every Banach $G$-module $W$ and every $n\geq 1$.\label{pt:van:int}
\item $\hb^n(G, E^*)=0$ for every $(G,X)$-module $E$ of type~\tm and every $n\geq 1$.\label{pt:van:M}
\item Any of the previous three points holds for $n=1$.\label{pt:van:1}
\item $C(X, V)\bid$ is relatively injective or every  Banach $G$-module $V$.\label{pt:inj:cont}
\item $\sI(C(X), W)^*$ is relatively injective or every  Banach $G$-module $W$.\label{pt:inj:int}
\item Every dual $(G,X)$-module of type~\tc is a relatively injective Banach $G$-module.\label{pt:inj:general}
\item There is a $G$-invariant element in $C(X, \ell^1 G)\bid$ summing to $1_X$.\label{pt:mean:norm}
\item There is a norm one positive $G$-invariant element in $C(X, \ell^1 G)\bid$ summing to $1_X$.\label{pt:mean:pos}
\end{enumerate}
\end{prop*}
\noindent
The only really new fact is the implication \eqref{pt:amen}$\Rightarrow$\eqref{pt:inj:general}, which still holds for $G$ locally compact second countable with the definitions of~\cite[\S\,2]{Anantharaman02} and~\cite[\S\,4.1]{Monod}.

\medskip

The summation condition in~\eqref{pt:mean:norm} and~\eqref{pt:mean:pos} means that the canonical image in $C(X)\bid$ of the invariant element is the bidual of the constant function $1_X$. The proof shows that it suffices that the summation be any function that never vanishes. The norm condition in~\eqref{pt:mean:pos} is redundant.

\section{Context}
This section presents more context than is needed for the proofs; we decided to follow a natural path rather than the shortest one. In particular, the discussion of crossed products can be ignored by the reader in a hurry. For definiteness, let us agree that Banach spaces and algebras are real. The few facts on operators and tensor products that we will use can all be found in Grothendieck~\cite{Grothendieck52,Grothendieck55}.

\subsection{$G$-modules and cohomology}
We will use almost nothing from the theory of bounded cohomology; for background, we refer to~\cite{Johnson,Gromov,Ivanov,Monod}. A \emph{Banach $G$-module} is a Banach space with an isometric linear $G$-representation, viewed also as a (left) $\ell^1 G$-module. Johnson~\cite{Johnson} allows also uniformly bounded modules; although this would not change anything to our arguments, we keep the isometric setting because exact norms are important when applying bounded cohomology to the study of characteristic classes~\cite{Gromov}. The Banach $G$-module $E$ is \emph{relatively injective} if any $G$-morphism $\alpha: A\to E$ from a Banach $G$-submodule $A\se B$ can be extended to $\beta: B\to E$ with $\|\beta\|=\|\alpha\|$ provided that there is a norm one projection $B\to A$ (as Banach spaces). This is equivalent to the existence of a $G$-equivariant norm one left inverse to the inclusion of $E$ into $\ell^\infty(G, E)$ where the latter in endowed with the diagonal $G$-action. (We will not use this equivalence.)

\subsection{$C(X)$-modules}
Given a compact space $X$, we endow the algebra $C(X)$ with the uniform norm. A (Banach) $C(X)$-module is an algebraic module $E$ such that $\norms{\fhi u}_E\leq \norms{\fhi}\cdot \norms{u}_E$ for $\fhi\in C(X)$, $u\in E$. For instance, consider the module of (strongly) continuous maps $E=C(X,V)$ to any Banach space $V$, with norm $\norms{\eta}_E=\sup_{x\in X}\norms{\eta(x)}_V$. If $V$ is also a Banach $G$-module, then $V$ becomes a Banach $G$-module for the diagonal action. Moreover, it is a $(G,X)$-module as defined in the introduction. Recalling our definition of the types~\tc and~\tm, one checks that the $C(X)$-module $C(X, E)$ is of type~\tc. More generally, if $Y$ is a compact space with a continuous map $Y\to X$, we obtain naturally a $C(X)$-module structure on $C(Y, E)$ and it is of type~\tc.

\begin{prop}
Let $E$ be a $C(X)$-module.
\begin{enumerate}
\item $E^*$ is of type~\tc if and only if $E$ is of type~\tm.\label{pt:dual:C}
\item $E^*$ is of type~\tm if and only if $E$ is of type~\tc.\label{pt:dual:M}
\item Both conditions \tc and \tm descend to submodules.\label{pt:CM:sub}
\item Each inequality implies the corresponding inequality for arbitrary $\fhi_i\in C(X)$ upon replacing $\fhi_i$ by $|\fhi_i|$ in the right hand side.\label{pt:CM:positive}
\end{enumerate}
\end{prop}

\begin{proof}
The definition of the dual norm implies readily that $E^*$ is of type~\tc (respectively~\tm) if $E$ is of type~\tm (resp.~\tc). The converse statements follow by duality after embedding $E$ into $E\bid$ and using point~\eqref{pt:CM:sub}, which is obvious. Point~\eqref{pt:CM:positive} follows from decomposing each $\fhi_i$ into positive and negative parts.
\end{proof}

\begin{cor}\label{cor:sI:M}
For any Banach space $W$, the $C(X)$-module $\sI(C(X), W)$ is of type~\tm.
\end{cor}

\begin{rem}\label{rem:sI:cont}
If $W$ happens to be a dual $W=V^*$, then  $\sI(C(X), W)\cong C(X,V)^*$; but in general $\sI(C(X), W)$ is not dual.
\end{rem}

\begin{proof}[Proof of Corollary~\ref{cor:sI:M}]
Recall that $\sI(C(X), W)$ sits isometrically as a closed subspace in the integral bilinear forms $\sBI(C(X)\times W^*)$. The latter is the dual of the injective tensor product $C(X)\itens W^*$ which identifies with $C(X, W^*)$. All this is compatible with the $C(X)$-structures. In other words, we realized $\sI(C(X), W)$ as a submodule of $C(X, W^*)^*$, which is of type~\tm.
\end{proof}

\subsection{The algebra $C(X, \ell^1 G)$}

An important particular case is $\sA:=C(X, \ell^1 G)$. There are canonical identifications
$$\sA:=C(X, \ell^1 G) \ \cong\ C(X)\itens \ell^1(G) \ \cong\ \uc{G}{C(X)}$$
where $\uc{G}{-}$ denotes unconditionally summable sequences indexed by $G$. Notice in particular that the \emph{finite} sums of the form $\sum_{g} \fhi_g\otimes \delta_g$ with $\fhi_g\in C(X)$ are dense in $\sA$. Moreover, the $\sA$-norm of such an element is
$$\Big\| \sum_{g} \fhi_g\otimes \delta_g \Big\|_{\sA} \ =\ \Big\| \sum_{g} |\fhi_g| \Big\|_{C(X)}$$
The $(G, X)$-structure turns $\sA$ into a Banach algebra that we call the \emph{Banach crossed product} of $G$ and $X$, not to be confused with the C*-algebraic crossed product. Explicitly, the product is $(\fhi\otimes\delta_g)(\psi\otimes\delta_h) = \fhi g\psi\otimes\delta_{gh}$ on elementary tensors with $g,h\in G$ and $\fhi, \psi\in C(X)$. The natural algebraic involution $\fhi\otimes\delta_g\mapsto g\inv\fhi\otimes \delta_{g\inv}$ does \emph{not} in general extend to $\sA$ (Example~\ref{ex:invol}), and there is indeed a fundamental difference between right and left $\sA$-modules, as we shall see. 

\begin{prop}\label{prop:gen:MC}
Let $G$ be a group acting on the compact space $X$ and let $E$ be a $(G, X)$-module. Let $\sA=C(X, \ell^1 G)$.
\begin{enumerate}
\item If $E$ is of type~\tc, then it is naturally a left $\sA$-module.
\item If $E$ is of type~\tm, then it is naturally a right $\sA$-module.
\end{enumerate}
\end{prop}

\begin{example}
The $(G, X)$-module $C(X)\ptens \ell^1 G$, where $\ptens$ is the projective tensor product, is not an $\sA$-module unless either $G$ or $X$ is finite. This is a special case of the Dvoretzky--Rogers Theorem~\cite{Dvoretzky-Rogers} since $C(X)\ptens \ell^1 G \cong \ell^1(G, C(X))$. In particular it is not of type~\tc; in addition, it is not of type~\tm either unless $X$ is a point.
\end{example}

\begin{proof}[Proof of Proposition~\ref{prop:gen:MC}]
The condition~(\tc) precisely shows that for finite sums $\alpha=\sum_g \alpha_g \otimes \delta_g$ and $u\in E$ we have
$$\norms{\alpha u}_E \ =\ \Big\| \sum_g \alpha_g (g u) \Big\|\ \leq\ \Big\| \sum_g |\alpha_g|\Big\|_{C(X)}\cdot \max_g\|g u\|_E\ =\ \|\alpha\|_{\sA}\cdot \|u\|_E$$
and thus the left multiplication extends to $\sA$. If $E$ is of type~\tm, we define
$$u \alpha:= \sum_g\big( (g\inv\alpha_g)\otimes\delta_{g\inv}\big) u\ =\  \sum_g g\inv(\alpha_g u),$$
which is a left module structure over finite sums. We then carry out a similar computation using the inequality~(\tm).
\end{proof}

\section{Characterisations of amenability}
The action of a group $G$ on a compact space $X$ is called \emph{topologically amenable} if there is a net $\{\mu^j\}_{j\in J}$ in $C(X, \ell^1 G)$ such that every $\mu^j(x)$ is a probability on $G$ and
$$\lim_{j\in J}\,  \big\| g\mu^j - \mu^j \big\|_{C(X, \ell^1 G)}\ =\ 0 \kern5mm(\forall\, g\in G).$$
Notice that if $G$ is countable, the net can be replaced by a sequence; this does not require $X$ to be metrisable and hence applies for instance to the \v{C}ech--Stone compactification of a countable group. Notice also that the continuity of $\mu^j:X\to \ell^1 G$ can be replaced by weak-* continuity since both topologies coincide on probabilities. By compactness of $X$, one can assume that the support of $\mu^j(x)$ is in a finite set depending only on $j$. Finally, we observe that, as expected, amenability is ``approximate properness'' with $\mu^j$ an approximate \emph{Bruhat function}.

Background references are~\cite{Anantharaman02,Brown-Ozawa}; we recall that the amenability of the action is equivalent to the nuclearity of the C*-algebraic (reduced) crossed product, and thus in turn to its amenability~\cite{Connes78,Haagerup83}.

\subsection{Proofs}
In order to take advantage of some simple implications, the proof will not follow a single cycle of implications. The only really new fact is the following implication, which still holds for $G$ locally compact second countable with the definitions of~\cite[\S\,2]{Anantharaman02} and~\cite[\S\,4.1]{Monod}.

\medskip

\eqref{pt:amen}$\Rightarrow$\eqref{pt:inj:general}.
Let $E$ be a $(G,X)$-module of type~\tc and consider the extension problem
$$
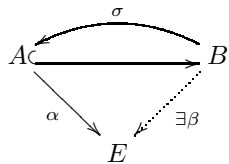
\xymatrix{
A \ar@<-0.75ex>@{^{(}->}[rr] \ar[dr]_{\alpha}&& B \ar@/_1pc/[ll]_{\sigma}\ar@{.>}[dl]^{\exists\beta} \\
&E & \\
}$$
where $\sigma$ is a norm one projection (which need not respect the module structure). Let $\mu\in C(G,\ell^1 G)$ be a finite sum $\sum_g \mu_g\otimes\delta_g$. We then define $\beta_\mu$ on each $b\in B$ by
$$\beta_\mu(b):= \sum_g (\mu_g\otimes\delta_g)\, \alpha(\sigma(g\inv b)).$$
In particular, $\beta_\mu$ is linear and $\beta|_A=\alpha$ if $\sum_g \mu_g = 1_X$. Moreover, the condition~\tc implies $\|\beta_\mu\|\leq \|\mu\|\cdot\|\alpha\|$. If now $E$ is a dual $E=V^*$, then those conditions are closed in the weak-* topology on $\sL(A, E)\cong (A\ptens V)^*$. Therefore, taking $\beta$ to be a weak-* accumulation point of the family $\beta_\mu$ where $\mu$ ranges over a net as in the definition of topological amenability, it remains only to show that $\beta$ is $G$-equivariant. For $h\in G$ and $b\in B$, we compute
$$\beta_\mu(h b) - h\beta_\mu(b)\ =\ \sum_g \Big(\mu_g - h\mu_{h\inv g}\Big)\, g\alpha(\sigma(g\inv h b))$$
which by condition~(\tc) is bounded in $E$-norm by
$$\big\|\mu - h\mu\big\|_{C(X, \ell^1 G)} \cdot \|\alpha\| \cdot \|b\|_B.$$
It follows indeed that $\beta$ is equivariant.

\medskip

\eqref{pt:inj:general}$\Rightarrow$\eqref{pt:van:M} and \eqref{pt:inj:cont}$\Rightarrow$\eqref{pt:van:cont}.
Both are due to the fact that cohomology with values in relatively injective modules vanishes in every positive degree (e.g. because the module is a trivial resolution of itself). For~\eqref{pt:van:M} we recall that $E^*$ is of type~\tc.

\medskip

\eqref{pt:inj:general}$\Rightarrow$\eqref{pt:inj:int}$\Rightarrow$\eqref{pt:inj:cont} and \eqref{pt:van:M}$\Rightarrow$\eqref{pt:van:int}$\Rightarrow$\eqref{pt:van:cont}$\Rightarrow$\eqref{pt:van:1}
are immediate in view of Corollary~\ref{cor:sI:M} and Remark~\ref{rem:sI:cont}.

\medskip

\eqref{pt:van:1}$\Rightarrow$\eqref{pt:mean:norm}.
Consider the exact sequence of Banach $G$-modules
$$0\lra C(X, \ell^1_0 G) \lra C(X, \ell^1 G) \lra C(X) \lra 0$$
where $\ell^1_0 G$ is the kernel of the summation map (as a Banach space sequence it is even split exact). Apply the bounded cohomology long exact sequence to the bidual sequence and use the vanishing assumption for $C(X, \ell^1_0 G)\bid$. This yields a short exact sequence of $G$-invariants
$$0 \lra \big(C(X, \ell^1_0 G)\bid\big)^G \lra  \big(C(X, \ell^1 G)\bid\big)^G \lra \big(C(X)\bid\big)^G \lra 0$$
and~\eqref{pt:mean:norm} follows.

\medskip

\eqref{pt:mean:norm}$\Rightarrow$\eqref{pt:amen}.
Let $\{\mu^j\}_{j\in J}$ be a net in $C(X, \ell^1 G)$ converging to an element as in~\eqref{pt:mean:norm}; we can assume that each $\mu^j$ ranges in the functions in $\ell^1 G$ that sum up to one. Since $g\mu^j - \mu^j$ converges weakly to zero for all $g\in G$, a standard Hahn--Banach argument shows that, upon passing to a net of finite convex combinations, we can assume that $\| g\mu^j - \mu^j\|$ converges to zero for all $g$. Recall that $C(X, \ell^1 G)$ is a Banach lattice and that therefore all lattice operations are (norm-) continuous. Moreover, they are $G$-equivariant. Recall further that canonical embeddings into biduals are compatible with all Banach lattice structures. In conclusion, we can further assume that all $\mu^j$ are non-negative, retaining all properties except that now the image $\Sigma\mu^j:=\sum_G\mu^j$ of $\mu^j$ in $C(X)$ is only \emph{bounded below} by $1_X$. However, the (equivariant) renormalization $\mu^j\mapsto \mu^j/\Sigma\mu^j$ is $2$-Lipschitz on maps with $\Sigma\mu^j\geq 1_X$ and therefore we finally obtain a net as required.

We could replace $1_X$ in the assumption by any never vanishing $f\in C(X)$; the Lipschitz constant would then be $2/\min_x |f(x)|$.

\medskip

\eqref{pt:amen}$\Rightarrow$\eqref{pt:mean:pos}
is obtained by choosing a weak-* accumulation point and keeping in mind the properties of Banach lattices recalled above. Since \eqref{pt:mean:pos}$\Rightarrow$\eqref{pt:mean:norm} is tautological, this completes the proof.

\section{Exact groups and various remarks}
The definition of exact groups goes back to~\cite{Kirchberg-Wassermann}; it is equivalent to the statement that the reduced C*-algebra of the group is exact for spatial tensor products. Ozawa~\cite{Ozawa_exact} and Anantharaman-Delaroche~\cite{Anantharaman02} proved that the group $G$ is exact if and only if its action on the \v{C}ech--Stone compactification $\beta G$ is amenable. Higson--Roe~\cite{Higson-Roe00} established that the latter condition is equivalent to Yu's Property~A. Exactness has deep implications for the group but at the same time it is notoriously difficult to produce any non-exact group; Gromov succeeded in~\cite{GromovRANDOM}.

\medskip

In view of the equivalence with amenability on $\beta G$, we have characterisations of exact groups. In this special case, $C(\beta G)\cong \ell^\infty G$, allowing some additional identifications. For instance, $(G, \beta G)$-modules are particularly concrete and can be thought of as modules over the involutive Banach--Hopf algebra $\ell^1 G$.

\begin{example}
Let $V$ be a Banach $G$-module. Any $G$-invariant $\ell^\infty G$-submodule of $\ell^\infty(G, V)$ is a $(G, \beta G)$-module of type~\tc. It is generally not of the form $C(\beta G, V)$.
\end{example}

This example provides a link to the work of Douglas--Nowak~\cite{Douglas-NowakVAN}. Another article investigating cohomological characterisations of exactness is~\cite{Brodzki-Niblo-Wright}.

\medskip

Another obvious example of $(G, \beta G)$-module is provided by the various $\ell^p G$.

\begin{example}\label{ex:GX}
Let $G$ be a non-Abelian free group. Then the $G$-action on $X:=\beta G$ is amenable. On the other hand, $\hb^2(G, \ell^1 G)$ is non-trivial, even though $\ell^1 G$ is a dual $(G, X)$-module. In fact, it is of type~\tm and is the dual of the $(G, X)$-module $c_0(G)$ of type~\tc. ($\hb^2(G, \ell^p G)$ is non-trivial for all $p<\infty$ using the argument of~\cite{Monod-ShalomCRAS}.)
\end{example}

The algebra
$$\sA:=C(\beta G, \ell^1 G)\ \cong\ \ell^\infty G \itens \ell^1 G\ \cong\ \uc{G}{\ell^\infty G},$$
which coincides with the algebra $\ell_u G$ of~\cite{Douglas-NowakVAN}, also admits additional realizations since now any element can be viewed as a kernel on $G\times G$. We find thus:
$$\sA\ \cong\ \sC(\ell ^1 G)\ \cong\ \sL_{\text{w*/w}}(\ell^\infty G),$$
where $\sC$ denotes compact operators and $\sL_{\text{w*/w}}$ the weak-*-to-weak continuous operators (which are necessarily compact)
. (The product, meanwhile, remains $G$-twisted).

\begin{example}\label{ex:invol}
The involution familiar from C*-crossed products is not $\sA$-bounded: if $S\se G$ is a finite set, the diagonal $\sum_{g\in S}\delta_g\otimes\delta_g$ has unit norm but its image $\delta_e\otimes\sum_{g\in S}\delta_{g\inv}$ has norm $|S|$.
\end{example}


\bibliographystyle{../BIB/amsalpha}
\bibliography{../BIB/ma_bib}
\end{document}